\documentclass[a4paper,reqno,11pt]{amsart}
\usepackage{amssymb}
\usepackage{amstext}
\usepackage{amsmath}
\usepackage{amscd}
\usepackage{amsthm}
\usepackage{amsfonts}
\usepackage{eepic}
\usepackage{enumerate}
\usepackage{graphicx}
\usepackage{latexsym}
\usepackage{mathrsfs}
\usepackage{slashbox}
\input xy
\xyoption{all}
\usepackage{pstricks}
\usepackage{lscape}
\usepackage{comment}

\newtheorem{theorem}{Theorem}[section]

\newtheorem{corollary}[theorem]{Corollary}
\newtheorem{lemma}[theorem]{Lemma}
\newtheorem{proposition}[theorem]{Proposition}
\newtheorem{mainthm}{Theorem}

\newtheorem{maincor}[mainthm]{Corollary}

\theoremstyle{definition}
\newtheorem{definition}[theorem]{Definition}
\newtheorem{remark}[theorem]{Remark}

\renewcommand{\S}{{\mathcal S}}

\newcommand{\add}{\mathsf{add}\hspace{.01in}}

\renewcommand{\mod}{\mathsf{mod}\hspace{.01in}}

\newcommand{\smod}{\underline{\mod}}

\newcommand{\arr}{\operatorname{1}\nolimits}

\newcommand{\cHom}{\overline{\Hom}}

\newcommand{\Ext}{\operatorname{Ext}\nolimits}
\newcommand{\Hom}{\operatorname{Hom}\nolimits}

\renewcommand{\Im}{\operatorname{Im}\nolimits}
\newcommand{\ind}{\operatorname{ind}\nolimits}

\newcommand{\m}{\operatorname{o}\nolimits}
\newcommand{\np}{\operatorname{np}\nolimits}

\newcommand{\s}{\operatorname{s}\nolimits}
\newcommand{\soc}{\operatorname{soc}\nolimits}

\newcommand{\ver}{\operatorname{0}\nolimits}

\newcommand{\trigid}{\mbox{\rm $\tau$-rigid}\hspace{.01in}}

\newcommand{\bLambda}{\overline{\Lambda}}

\newcommand{\Ma}{\Delta}
\newcommand{\Mat}{\Delta(\Lambda)}

\newcommand{\xto}{\xrightarrow}
\newcommand{\kD}{D}
\begin{document}
\title[Characterizing $\tau$-rigid-finite algebras with radical square zero]{Characterizing $\tau$-rigid-finite algebras with radical square zero}
\author[T.Adachi]{Takahide Adachi}\thanks{
The author wishes to express his sincere gratitude to Professor Osamu Iyama for suggesting these problems and valuable advice.
He is grateful to Takuma Aihara for helpful discussions.
He thanks Yuya Mizuno for useful comments.
}
\address{Graduate School of Mathematics, Nagoya University, Frocho, Chikusaku, Nagoya, 464-8602, Japan}
\email{m09002b@math.nagoya-u.ac.jp}
\begin{abstract}
In this paper, we give a characterization of $\tau$-rigid-finite algebras with radical square zero in terms of the separated quivers,
which is an analog of a famous characterization of representation-finite algebras with radical square zero due to Gabriel.
\end{abstract}
\maketitle
\section{Introduction}
In 1980's, Auslander-Smalo \cite{AS} (see also \cite{Sk, ASS}) studied a class of modules, 
which are called $\tau$-rigid nowadays (see Definition \ref{2.1}), over finite dimensional algebras.
From the perspective of tilting mutation theory, 
the authors in \cite{AIR} introduced the notion of (support) $\tau$-tilting modules as a special class of $\tau$-rigid modules.
They correspond bijectively with many important objects in representation theory, 
{\it i.e.}, functorially finite torsion classes, two-term silting complexes and cluster-tilting objects in a special cases.
Therefore it is important to classify $\tau$-rigid-finite algebras, that is, algebras having
only finitely many indecomposable $\tau$-rigid modules, or equivalently basic support $\tau$-tilting modules (see \cite{DIJ}).
Recently, $\tau$-rigid modules are studied by several authors 
(Jasso \cite{Ja}, Mizuno \cite{Mi} and Malicki-de la Pe$\tilde{\mathrm{n}}$a-Skowro$\acute{\mathrm{n}}$ski \cite{MPS}) and
in particular, various aspects of $\tau$-rigid modules over algebras with radical square zero are studied by Zhang \cite{Zh}.

In this paper, we study $\tau$-rigid modules over algebras with radical square zero, 
which provide one of the most fundamental classes of algebras in representation theory 
({\it e.g.}, work of Yoshii \cite{Yo} in 1956 and Gabriel \cite{Ga} in 1972).
For an algebra $\Lambda$ with radical square zero, an important role
is played by a stable equivalence $F:\smod \Lambda \rightarrow \smod KQ^{\s}$,
where $Q^{\s}$ is the separated quiver for $\Lambda$, which is defined by the vertex set $Q^{\s}_{\ver}=\{ i^{+},i^{-}\mid i\in Q_{\ver}\}$ 
and the arrow set $Q^{\s}_{\arr}=\{ i^{+}\rightarrow j^{-}\mid i\rightarrow j\ {\textnormal{in}}\ Q_{\arr}\}$ 
for the vertex set $Q_{0}$ and the arrow set $Q_{1}$ of the quiver $Q$ for $\Lambda$.
In particular, we have the following famous theorem characterizing representation-finiteness.
\begin{mainthm}\cite{Ga, ARS}
Let $\Lambda$ be a finite dimensional algebra with radical square zero.
Then the following are equivalent:
\begin{itemize}
\item[(1)] $\Lambda$ is representation-finite.
\item[(2)] The separated quiver for $\Lambda$ is a disjoint union of Dynkin quivers.
\end{itemize}
\end{mainthm}

The following main theorem of this paper is an analog of this result for $\tau$-rigid-finiteness.
A full subquiver of $Q^{\s}$ is called a {\it single subquiver} if, for any $i\in Q_{\ver}$, the vertex set contains at most one of $i^{+}$ or $i^{-}$.
\begin{mainthm}[Theorem \ref{3.1}]\label{main}
Let $\Lambda$ be a finite dimensional algebra with radical square zero.
Then the following are equivalent:
\begin{itemize}
\item[(1)] $\Lambda$ is $\tau$-rigid-finite.
\item[(2)] Every single subquiver of the separated quiver for $\Lambda$ is a disjoint union of Dynkin quivers. 
\end{itemize}
\end{mainthm}

The following result plays a crucial role in the proof of Theorem \ref{main}.
\begin{mainthm}[Theorem \ref{3.2}]
Let $X$ be a $\Lambda$-module.
Let $P_{1}^{X}\xto{}P_{0}^{X}\xto{}X\xto{}0$ be a minimal projective presentation of $X$.
The following are equivalent:
\begin{itemize}
\item[(1)] $X$ is a $\tau$-rigid $\Lambda$-module.
\item[(2)] $FX$ is a $\tau$-rigid $KQ^{\s}$-module and $\add P_{0}^{X}\cap \add P_{1}^{X}=0$.
\end{itemize}
\end{mainthm}

Moreover, we have the following bijection.
For an algebra $\Lambda$, we denote by $\trigid\Lambda$ the set of isomorphism classes of indecomposable $\tau$-rigid $\Lambda$-modules
and let $\trigid^{\m}\Lambda:=\{ X\in\trigid\Lambda \mid \add(P_{0}^{X}\oplus P_{1}^{X})=\add \Lambda \}$, 
where $P_{1}^{X}\xto{}P_{0}^{X}\xto{}X\xto{}0$ is a minimal projective presentation of $X$.  
We denote by $\S^{+}$ the set of all connected single subquivers of $Q^{\s}$ except the quivers with exactly one vertex $i^{-}$ for any $i\in Q_{0}$.
\begin{maincor}[Corollary \ref{3.5}]
There is a natural bijection
\begin{align}
\trigid\Lambda\longrightarrow\coprod_{Q^{\prime}\in\S^{+}}\trigid^{\m} KQ^{\prime}. \notag
\end{align}
\end{maincor}

As an application, we have the following results.
First, we give a positive answer to a question given by Zhang \cite{Zh}.
Namely, for an algebra $\Lambda$ with radical square zero, if every indecomposable $\Lambda$-module is $\tau$-rigid, then $\Lambda$ is representation-finite.
Secondly, we give an example of $\tau$-rigid-finite algebras which is not representation-finite. 
Let $\Lambda$ be a multiplicity-free Brauer cyclic graph algebra with $n$ vertices.
Then $\Lambda$ is not representation-finite.
Moreover, it is $\tau$-rigid-finite if and only if $n$ is odd.
In this case, the cardinality of $\trigid\Lambda$ is 
\begin{align}
|\trigid\Lambda|=2n^{2}-n. \notag
\end{align}

Throughout this paper, we use the following notation.
By an algebra we mean basic and finite dimensional algebra over an algebraically closed field $K$, 
and by a module we mean a finite dimensional right module.
For an algebra $\Lambda$, we denote by $\mod \Lambda$ the category of $\Lambda$-modules, by $\smod\Lambda$ the stable category,
and by $\tau$ the Auslander-Reiten translation of $\Lambda$.
We call a quiver {\it Dynkin} (respectively, {\it Euclidean}) 
if the underlying graph is one of Dynkin (respectively, Euclidean) graphs of type $A,D$ 
and $E$ (respectively, $\widetilde{A},\widetilde{D}$ and $\widetilde{E}$).
We refer to \cite{ASS,ARS} for background on representation theory.
\section{Preliminaries}
In this section, we collect some results which are necessary in this paper.
Let $\Lambda$ be an algebra and $J:=J_{\Lambda}$ a Jacobson radical of $\Lambda$.
For a $\Lambda$-module $X$, we denote by  
\begin{align}
P_{1}^{X}\overset{p}{\rightarrow}P_{0}^{X}\overset{q}{\rightarrow}X\rightarrow 0 \notag 
\end{align}
a minimal projective presentation of $X$.
\subsection{$\tau$-rigid modules}\label{subsec2.1}
We recall basic properties of $\tau$-rigid modules.
\begin{definition}\label{2.1}
A $\Lambda$-module $X$ is called {\it $\tau$-rigid} if $\Hom_{\Lambda}(X,\tau X)=0$.
We denote by $\trigid \Lambda$ the set of isomorphism classes of indecomposable $\tau$-rigid $\Lambda$-modules.
An algebra $\Lambda$ is called {\it $\tau$-rigid-finite} if $\trigid \Lambda$ is a finite set.
\end{definition}
By Auslander-Reiten duality $\Ext_{\Lambda}^{1}(X,Y)\simeq \kD\cHom_{\Lambda}(Y,\tau X)$, 
every $\tau$-rigid $\Lambda$-module $X$ is {\it rigid} ({\it i.e.} $\Ext_{\Lambda}^{1}(X,X)=0$), and
the converse is true if $\Lambda$ is hereditary ({\it e.g.}, path algebra $KQ$ of an acyclic quiver $Q$).

The following proposition plays an important role in this paper.
\begin{proposition}\cite[Proposition 2.4 and 2.5]{AIR}\label{2.2}
For a $\Lambda$-module $X$, the following hold.
\begin{itemize}
\item[(1)] $X$ is $\tau$-rigid if and only if the map $(p,X):\Hom_{\Lambda}(P_{0}^{X},X)\rightarrow \Hom_{\Lambda}(P_{1}^{X},X)$ is surjective. 
\item[(2)] If $X$ is $\tau$-rigid, then we have $\add P_{0}^{X}\cap \add P_{1}^{X}=0$.
\end{itemize}
\end{proposition}
For an idempotent $e\in \Lambda$, we consider two $K$-linear functors
\begin{align}
L_{e}(-):=-\otimes_{e\Lambda e}e\Lambda: \mod (e\Lambda e)\longrightarrow \mod \Lambda,\ \ 
R_{e}(-):=(-)e:\mod \Lambda\longrightarrow \mod (e\Lambda e). \notag
\end{align}
Then $(L_{e},R_{e})$ is an adjoint pair.
Moreover, the following result gives a  connection between $\tau$-rigid $(e\Lambda e)$-modules and $\tau$-rigid $\Lambda$-modules.
\begin{lemma}\cite[I.6.8]{ASS}\label{2.3}
Let $\Lambda$ be an algebra and $e\in \Lambda$ an idempotent.
\begin{itemize}
\item[(1)] The functor $L_{e}$ is fully faithful and there exists a functorial isomorphism $R_{e} L_{e}\simeq 1_{\mod eAe}$.
In particular, $L_{e}$ and $R_{e}$ induce mutually quasi-inverse equivalences 
between categories $\mod(e\Lambda e)$ and $\Im L_{e}:=\{ L_{e}(X)\mid X\in\mod(e\Lambda e)\}$.
\item[(2)] A $\Lambda$-module $X$ is in the category $\Im L_{e}$ if and only if $P_{0}^{X}\oplus P_{1}^{X}\in \add e\Lambda$.
\end{itemize}
\end{lemma}
We have the following result.
\begin{proposition}\label{2.4}
Let $\Lambda$ be an algebra and $e\in\Lambda$ an idempotent.
Assume that a $\Lambda$-module $X$ is in $\Im L_{e}$.
Then $X$ is $\tau$-rigid if and only if the $(e\Lambda e)$-module $Xe$ is $\tau$-rigid.
In particular, $L_{e}$ and $R_{e}$ induce mutually inverse bijections
\begin{align}
\trigid(e\Lambda e)\longleftrightarrow \trigid \Lambda \cap \Im L_{e}.\notag
\end{align}
\end{proposition}
\begin{proof}
By Lemma \ref{2.3}(2), we have $P_{0}^{X}\oplus P_{1}^{X}\in \add e\Lambda$. Hence the sequence 
\begin{align}
P_{1}^{X}e \xto{pe}P_{0}^{X}e\xto{qe} Xe \xto{} 0 \notag
\end{align}
is a projective presentation.
By Lemma \ref{2.3}(1), the projective presentation is minimal, 
and moreover we have a commutative diagram 
\begin{align}
\xymatrix{
\Hom_{A}(P_{0}^{X},X)\ar[r]^{(p,X)}\ar[d]^{\simeq}&\Hom_{A}(P_{1}^{X},X)\ar[d]^{\simeq}\\
\Hom_{eAe}(P_{0}^{X}e,Xe)\ar[r]^{(pe,Xe)}&\Hom_{eAe}(P_{1}^{X}e,Xe)
}\notag
\end{align}
where the vertical maps are isomorphisms.
By using Proposition \ref{2.2}(1), we have that 
$X$ is a $\tau$-rigid $\Lambda$-module if and only if $Xe$ is a $\tau$-rigid $(e\Lambda e)$-module.
\end{proof}
The following proposition is a well-known result for path algebras.
\begin{proposition}\cite[VII.5.10, VIII.2.7 and VIII.2.9]{ASS}\label{2.5}
Let $Q$ be a connected acyclic quiver and $\Lambda:=KQ$ the path algebra of $Q$.
Then the following hold.
\begin{itemize}
\item[(1)] $\Lambda$ is representation-finite if and only if $Q$ is a Dynkin quiver.
In this case, every indecomposable $\Lambda$-module is rigid.
\item[(2)] If $\Lambda$ is not representation-finite, 
then there exist infinitely many isomorphism classes of indecomposable rigid $\Lambda$-modules.
Moreover, there exists an indecomposable $\Lambda$-module which is not rigid.
\end{itemize}
\end{proposition}
Immediately, we have the following characterization of $\tau$-rigid-finiteness for path algebras of acyclic quivers.
\begin{theorem}\label{2.6}
Let $Q$ be a connected acyclic quiver and $\Lambda:=KQ$ the path algebra of $Q$.
Then the following are equivalent:
\begin{itemize}
\item[(1)] $\Lambda$ is representation-finite.
\item[(2)] $\Lambda$ is $\tau$-rigid-finite.
\item[(3)] $Q$ is a Dynkin quiver.
\end{itemize}
\end{theorem}
\begin{proof}
It follows from Proposition \ref{2.5} 
because rigid modules are exactly $\tau$-rigid modules for any hereditary algebra.
\end{proof}
\subsection{Algebras with radical square zero}\label{subsec2.2}
Throughout this subsection, we assume that $\Lambda$ is an algebra with radical square zero ({\it i.e.}, $J^{2}=0$).
For algebras with radical square zero, the following triangular matrix algebra plays an important role:
\begin{align}
\Ma :=\Mat :=
\left[\begin{smallmatrix}
\Lambda/J&J\\
0&\Lambda/J
\end{smallmatrix}\right].\notag
\end{align}
Each $\Ma$-module is given by a triplet $(X^{\prime},X^{\prime\prime};\varphi)$, 
where $X^{\prime}, X^{\prime\prime}$ are $(\Lambda/J)$-modules and $\varphi$ is a morphism
\begin{align}
\varphi:X^{\prime}\otimes_{\Lambda/J}J\xto{}X^{\prime\prime}\notag
\end{align} 
in $\mod(\Lambda/J)$. 
A morphism $f:(X^{\prime},X^{\prime\prime};\varphi)\xto{}(Y^{\prime},Y^{\prime\prime};\psi)$ in $\mod\Ma$ is given by a pair $(f^{\prime},f^{\prime\prime})$, 
where $f^{\prime}:X^{\prime}\xto{}Y^{\prime}$ and $f^{\prime\prime}:X^{\prime\prime}\xto{}Y^{\prime\prime}$ are morphisms in $\mod(\Lambda/J)$ 
such that $\psi f^{\prime}=f^{\prime\prime}\varphi$ (see \cite[A.2.7]{ASS} and \cite[III.2]{ARS} for details).
\begin{align}
\xymatrix{
X^{\prime}\otimes_{\Lambda/J}J\ar[r]^{\ \ \ \ \varphi}\ar[d]_{f^{\prime}\otimes J}&X^{\prime\prime}\ar[d]^{f^{\prime\prime}}\\
Y^{\prime}\otimes_{\Lambda/J}J\ar[r]^{\ \ \ \ \psi}&Y^{\prime\prime}}\notag
\end{align}

We recall some properties of the triangular matrix algebra $\Ma$.
Let $Q=(Q_{\ver},Q_{\arr})$ be a quiver, where $Q_{\ver}$ is the vertex set and $Q_{\arr}$ is the arrow set.
Then we define a new quiver $Q^{\s}=(Q_{\ver}^{\s},Q_{\arr}^{\s})$, called a {\it separated quiver}, as follows:
Let $Q^{+}_{\ver}:=\{ i^{+}\mid i\in Q_{\ver}\}$ and  $Q^{-}_{\ver}:=\{ i^{-}\mid i\in Q_{\ver}\}$ be copies of $Q_{\ver}$. 
Then
\begin{align}
Q_{\ver}^{\s}:=Q_{\ver}^{+}\coprod Q_{\ver}^{-},\ \ Q_{\arr}^{\s}:=\{ i^{+}\rightarrow j^{-} \mid i\rightarrow j\ \textnormal{in}\  Q_{\arr}\}. \notag
\end{align}
Note that the separated quiver $Q^{\s}$ is not necessarily connected even if $Q$ is connected.
For example, the separated quiver $Q^{\s}$ of the following quiver $Q$ is not connected:
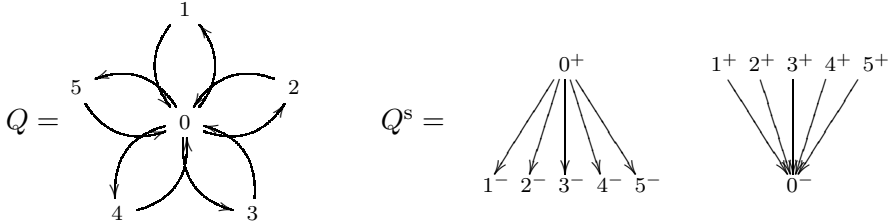
\begin{align}
\begin{xy}
(-20,0) *{Q=}="Z",
(30,0) *{Q^{\s}=}="Z",
(0,0) *+{\begin{smallmatrix}0\end{smallmatrix}}="A",
(0,15) *+{\begin{smallmatrix}1\end{smallmatrix}}="B",
(14.3,4.6) *+{\begin{smallmatrix}2\end{smallmatrix}}="C",
(8.9,-12.1) *+{\begin{smallmatrix}3\end{smallmatrix}}="D",
(-8.9,-12.1) *+{\begin{smallmatrix}4\end{smallmatrix}}="E",
(-14.3,4.6) *+{\begin{smallmatrix}5\end{smallmatrix}}="F",
(50,8) *+{\hspace{2mm}\begin{smallmatrix}0^{+}\end{smallmatrix}}="G",
(40,-8) *{\hspace{2mm}\begin{smallmatrix}1^{-}\end{smallmatrix}}="H",
(45,-8) *{\hspace{2mm}\begin{smallmatrix}2^{-}\end{smallmatrix}}="I",
(50,-8) *{\hspace{2mm}\begin{smallmatrix}3^{-}\end{smallmatrix}}="J",
(55,-8) *{\hspace{2mm}\begin{smallmatrix}4^{-}\end{smallmatrix}}="K",
(60,-8) *{\hspace{2mm}\begin{smallmatrix}5^{-}\end{smallmatrix}}="L",
(80,-8) *{\hspace{2mm}\begin{smallmatrix}0^{-}\end{smallmatrix}}="M",
(70,8) *+{\hspace{2mm}\begin{smallmatrix}1^{+}\end{smallmatrix}}="N",
(75,8) *+{\hspace{2mm}\begin{smallmatrix}2^{+}\end{smallmatrix}}="O",
(80,8) *+{\hspace{2mm}\begin{smallmatrix}3^{+}\end{smallmatrix}}="P",
(85,8) *+{\hspace{2mm}\begin{smallmatrix}4^{+}\end{smallmatrix}}="Q",
(90,8) *+{\hspace{2mm}\begin{smallmatrix}5^{+}\end{smallmatrix}}="R",
\ar@/^-4mm/ "A";"B"
\ar@/^-4mm/ "B";"A"
\ar@/^-4mm/ "A";"C"
\ar@/^-4mm/ "C";"A"
\ar@/^-4mm/ "A";"D"
\ar@/^-4mm/ "D";"A"
\ar@/^-4mm/ "A";"E"
\ar@/^-4mm/ "E";"A"
\ar@/^-4mm/ "A";"F"
\ar@/^-4mm/ "F";"A"
\ar "G";"H"
\ar "G";"I"
\ar "G";"J"
\ar "G";"K"
\ar "G";"L"
\ar "N";"M"
\ar "O";"M"
\ar "P";"M"
\ar "Q";"M"
\ar "R";"M"
\end{xy}\notag
\end{align}
We call a quiver {\it bipartite} if each vertex is either a sink or a source.
\begin{proposition}\cite[III.2.5]{ARS}
Let $Q$ be the quiver of $\Lambda$.
The following hold.
\begin{itemize}
\item[(1)] The separated quiver $Q^{\s}$ is bipartite.
\item[(2)] The algebra $\Ma$ is isomorphic to the path algebra of $Q^{\s}$. 
In particular, $\Ma$ is a hereditary algebra with radical square zero.
\item[(3)] Each simple $\Ma$-module is one of the form $(S,0;0)$ or $(0,S;0)$, where $S$ is a simple $\Lambda$-module.
\item[(4)] Each indecomposable projective $\Ma$-module is one of the form $(P/PJ, PJ ; 1_{PJ})$ or $(0,P/PJ;0)$, 
where $P$ is an indecomposable projective $\Lambda$-module.
\end{itemize}
\end{proposition}
Now we recall results on representation theory of algebras with radical square zero.
We define a functor $F:\mod \Lambda\xto{} \mod\Ma$ as follows:
For any $\Lambda$-module $X$, we let 
\begin{align}
F(X):=(X/XJ, XJ; \varphi_{X}),\notag
\end{align} 
where the map $\varphi_{X}:X/XJ\otimes_{\Lambda/J} J\xto{} XJ$ is naturally induced by 
the natural multiplication morphism $X\otimes_{\Lambda} J\xto{} XJ$ since $J^{2}=0$.
For any morphism $g:X\xto{} Y$, we let 
\begin{align}
F(g):=(g^{\prime},g^{\prime\prime}), \notag
\end{align} 
where $g^{\prime}:X/XJ\xto{} Y/YJ$ is induced by $g$ and $g^{\prime\prime}:XJ\xto{} YJ$ is the restriction to $XJ$.
\begin{proposition}\cite[X.2.1, X.2.2, X.2.4 and X.2.6]{ARS}\label{2.7}
The following hold.
\begin{itemize}
\item[(1)] The functor $F$ is full and induces an equivalence of categories $\smod\Lambda\rightarrow \smod\Ma$.
\item[(2)] A $\Lambda$-module $X$ is indecomposable (respectively, projective) if and only if $FX$ is an indecomposable (respectively, a projective) $\Ma$-module.
\item[(3)] The following are equivalent:
\begin{itemize}
\item[(a)] $\Lambda$ is representation-finite. 
\item[(b)] The separated quiver of the quiver for $\Lambda$ is a disjoint union of Dynkin quivers.
\end{itemize}
\end{itemize}
\end{proposition}
\begin{remark}
Clearly a stable equivalence preserves representation-finiteness. 
However a stable equivalence does not preserve $\tau$-rigid-finiteness in general.
Indeed, let $\Lambda$ be an algebra with radical square zero whose quiver consists of one vertex and $n$ loops with $n\ge 2$.
Since $\Lambda$ is local, every indecomposable $\tau$-rigid $\Lambda$-module is projective, and in particular $\Lambda$ is $\tau$-rigid-finite.
On the other hand, since the separated quiver is the $n$-Kronecker quiver
\begin{align}
\xymatrix{
\circ\ar@<2.6mm>[r]\ar@<-2.4mm>[r]^{{\tiny\vdots}}&\circ
},\notag
\end{align}
$\Ma$ is not $\tau$-rigid-finite by Theorem \ref{2.6}.
\end{remark}
\section{Main results}
Throughout this section, $\Lambda$ is an algebra with radical square zero ({\it i.e.}, $J^{2}=0$), and $\Ma,F$ as in Subsection \ref{subsec2.2}.
Let $Q$ be the quiver of $\Lambda$ and $Q^{\s}$ the separated quiver of $Q$.
A full subquiver $Q^{\prime}$ of $Q^{\s}$ is called a {\it single subquiver}
if, for any $i\in Q_{\ver}$, the vertex set $Q^{\prime}_{\ver}$ contains at most one of $i^{+}$ or $i^{-}$.
We denote by $\S$ the set of all single subquivers of $Q^{\s}$.

The following theorem is our main result of this paper.
\begin{theorem}\label{3.1}
Let $\Lambda$ be an algebra with radical square zero and $Q^{\s}$ the separated quiver of the quiver $Q$ for $\Lambda$. 
Then the following are equivalent:
\begin{itemize}
\item[(1)] $\Lambda$ is $\tau$-rigid-finite.
\item[(2)] Every single subquiver of $Q^{\s}$ is a disjoint union of Dynkin quivers.
\end{itemize}
\end{theorem}
The proof of Theorem \ref{3.1} will be given in the rest of this section.
A key result is the following criterion for indecomposable $\Lambda$-modules to be $\tau$-rigid in terms of the triangular matrix algebra $\Ma$.
\begin{theorem}\label{3.2}
Let $X$ be a $\Lambda$-module.
The following are equivalent:
\begin{itemize}
\item[(1)] $X$ is a $\tau$-rigid $\Lambda$-module.
\item[(2)] $FX$ is a $\tau$-rigid $\Ma$-module and $\add P_{0}^{X}\cap \add P_{1}^{X}=0$.
\end{itemize}
\end{theorem}
We prove Theorem \ref{3.2} by comparing a minimal projective presentation of a $\Lambda$-module $X$ with that of the $\Ma$-module $FX$.
We start with an easy lemma.
\begin{lemma}\label{3.3}
Let $f:X\rightarrow Y$ be a nonzero morphism in $\mod\Lambda$.
If $\Im f$ is contained in $YJ$, then there exists a unique morphism $\tilde{f}: X/XJ\rightarrow YJ$ such that 
\begin{align}
f=(X\xto{\pi}X/XJ\xto{\tilde{f}}YJ\xto{\iota}Y), \notag
\end{align}
where $\pi$ and $\iota$ are natural morphisms.
\end{lemma}
\begin{proof}
This is clear since $f(XJ)\subseteq (YJ)J=0$ holds by $J^{2}=0$.
\end{proof}
The following lemma gives a construction of a minimal projective presentation of $FX$ from that of $X$.
\begin{lemma}\label{3.4}
Let $P_{1}^{X}\overset{p}{\rightarrow} P_{0}^{X}\overset{q}{\rightarrow} X \rightarrow 0$ be a minimal projective presentation of a $\Lambda$-module $X$.
Then 
\begin{align}\label{eq1}
\xymatrix{
0\ar[r]&(0,P_{1}^{X}/P_{1}^{X}J;0)\ar[r]^{\ \ \ \ \ (0,\tilde{p})}&FP^{X}_{0}\ar[r]^{Fq}&FX\ar[r]&0
}\tag{$\ast$}
\end{align}
is a minimal projective resolution of the $\Ma$-module $FX$.
\end{lemma}
\begin{proof}
Since $J^{2}=0$ holds, $\ker q$ is semisimple. 
Hence $\ker q=P_{1}^{X}/P_{1}^{X}J$ holds.
Since $\ker q$ is contained in $P_{0}^{X}J$, by Lemma \ref{3.3}, we have a decomposition 
\begin{align}
p=(P_{1}^{X}\xto{\pi}P_{1}^{X}/P_{1}^{X}J\xto{\tilde{p}}P_{0}^{X}J\xto{\iota}P_{0}^{X}),\notag
\end{align}
where $\pi$ and $\iota$ are natural morphisms.
Thus, we have the following commutative diagram 
\begin{align}
\xymatrix@R=6mm{
&0\ar[d]&0\ar[d]&&\\
&P_{1}^{X}/P_{1}^{X}J\ar@{=}[r]\ar[d]^{\tilde{p}}&\ker q\ar[d]&&\\
0\ar[r]&P_{0}^{X}J\ar[r]^{\iota}\ar[d]^{q^{\prime\prime}}&P_{0}^{X}\ar[r]\ar[d]^{q}&P_{0}^{X}/P_{0}^{X}J\ar[r]\ar@{=}[d]^{q^{\prime}}&0\\
0\ar[r]&XJ\ar[r]\ar[d]&X\ar[r]\ar[d]&X/XJ\ar[r]&0\\
&0&0&&
}\notag
\end{align}
where $Fq=(q^{\prime},q^{\prime\prime})$.
Thus the exact sequence
\begin{align}
0\xto{}(0,P_{1}^{X}/P_{1}^{X}J;0) \xto{(0,\tilde{p})}(P_{0}^{X}/P_{0}^{X}J, P_{0}^{X}J;1_{P^{X}_{0}J})\xto{(q^{\prime},q^{\prime\prime})}(X/XJ,XJ;\varphi_{X})\xto{}0 \notag
\end{align}
in $\mod \Ma$ is a minimal projective resolution of $FX$, because $\Ma$ is hereditary and $(0,\tilde{p})$ is in the radical of $\mod\Lambda$.
\end{proof}
\begin{proof}[Proof of Theorem \ref{3.2}]
Let $\iota:P_{0}^{X}J\rightarrow P_{0}^{X}$, $\iota^{\prime}: XJ\rightarrow X$ and $\pi: P_{1}^{X}\rightarrow P_{1}^{X}/P_{1}^{X}J$ are natural morphisms.

(1)$\Rightarrow$(2):
Assume that $X$ is $\tau$-rigid.
By Proposition \ref{2.2}(2), we have $\add P_{0}^{X}\cap \add P_{1}^{X}=0$.
Now we show that $FX$ is a $\tau$-rigid $\Ma$-module.
We have a minimal projective resolution (\ref{eq1}) in Lemma \ref{3.4}.
By Proposition \ref{2.2}(1), we have only to show that 
\begin{align}\label{eq2}
((0,\tilde{p}), FX):\Hom_{\Ma}(FP,FX)\rightarrow \Hom_{\Ma}((0,P_{1}^{X}/P_{1}^{X}J;0),FX) \tag{$\natural$}
\end{align}
is surjective.
Let 
\begin{align}
(0,h):(0,P_{1}^{X}/P_{1}^{X}J;0)\rightarrow FX=(X/XJ, XJ;\varphi_{X})\notag
\end{align}
be a morphism in $\mod\Ma$ and $f:=\iota^{\prime} h \pi:P_{1}^{X}\rightarrow X$ a morphism in $\mod\Lambda$.
Since $X$ is $\tau$-rigid, there exists a morphism $g:P_{0}^{X}\rightarrow X$ such that $f=gp$.
Let 
\begin{align}
Fg:=(g^{\prime},g^{\prime\prime}):(P_{0}^{X}/P_{0}^{X}J,P_{0}^{X}J; 1_{P_{0}^{X}J})\rightarrow (X/XJ, XJ;\varphi_{X}).\notag
\end{align}
Then we have 
\begin{align}
\iota^{\prime} h \pi=f=gp=g\iota\tilde{p}\pi=\iota^{\prime} g^{\prime\prime}\tilde{p}\pi, \notag
\end{align}
and hence we have $h=g^{\prime\prime}\tilde{p}$.
Thus we have $(0,h)=(g^{\prime},g^{\prime\prime})(0,\tilde{p})$.
Consequently, the map (\ref{eq2}) is surjective.
\begin{align}
\xymatrix@R=5mm{P_{1}^{X}\ar[dd]^{f}\ar[rd]^{p}\ar[rrr]^{\pi}&&&P_{1}^{X}/P_{1}^{X}J\ar[dd]^{h}\ar[dl]^{\tilde{p}}\\
&P_{0}^{X}\ar[dl]^{g}&P_{0}^{X}J\ar[dr]^{g^{\prime\prime}}\ar[l]^{\iota}&\\
X&&&XJ\ar[lll]^{\iota^{\prime}}
}\notag
\end{align}

(2)$\Rightarrow$(1):
Assume that $FX$ is $\tau$-rigid and $\add P_{0}^{X}\cap \add P_{1}^{X}=0$.
We have only to show that 
\begin{align}\label{eq3}
(p,X):\Hom_{\Lambda}(P_{0}^{X},X) \rightarrow \Hom_{\Lambda}(P_{1}^{X},X) \tag{$\sharp$}
\end{align}
is surjective by Proposition \ref{2.2}(1).
Let $f:P_{1}^{X}\rightarrow X$ be a morphism in $\mod\Lambda$.
Since $\add P_{0}^{X}\cap \add P_{1}^{X}=0$ holds, $\Im f$ is contained in $XJ$.
Thus, by Lemma \ref{3.3}, 
there exists a morphism $\tilde{f}:P_{1}^{X}/P_{1}^{X}J\rightarrow XJ$ in $\mod\Lambda$ such that $f=\iota^{\prime} \tilde{f}\pi$.
Now we consider a morphism
\begin{align}
(0,\tilde{f}):(0,P_{1}^{X}/P_{1}^{X}J;0)\xto{}FX=(X/XJ,XJ;\varphi_{X}) \notag
\end{align}
in $\mod\Ma$.
Since (\ref{eq1}) in Lemma \ref{3.4} gives a minimal projective resolution and $FX$ is $\tau$-rigid, 
there exists a morphism $(g^{\prime}, g^{\prime\prime}):FP_{0}^{X}\rightarrow FX$ in $\mod\Ma$ 
such that $(0,\tilde{f})=(g^{\prime},g^{\prime\prime})(0,\tilde{p})$.
In particular, we have $\tilde{f}=g^{\prime\prime}\tilde{p}$.
Since $F$ is full by Proposition \ref{2.7}(1), there exists a morphism $g:P_{0}^{X}\rightarrow X$ such that $Fg=(g^{\prime},g^{\prime\prime})$.
Then $g^{\prime\prime}$ is a restriction of $g$ by construction of $F$, and we have
\begin{align}
gp=g\iota\tilde{p}\pi=\iota^{\prime}g^{\prime\prime}\tilde{p}\pi=\iota^{\prime}\tilde{f}\pi=f. \notag
\end{align}
Consequently, the map (\ref{eq3}) is surjective.
\begin{align}
\xymatrix@R=5mm{P_{1}^{X}\ar[dd]^{f}\ar[rd]^{p}\ar[rrr]^{\pi}&&&P_{1}^{X}/P_{1}^{X}J\ar[dd]^{\tilde{f}}\ar[dl]^{\tilde{p}}\\
&P_{0}^{X}\ar[dl]^{g}&P_{0}^{X}J\ar[dr]^{g^{\prime\prime}}\ar[l]^{\iota}&\\
X&&&XJ\ar[lll]^{\iota^{\prime}}
}\notag
\end{align}
This finishes the proof. 
\end{proof}
For any indecomposable $\Lambda$-module $X$, 
we decompose the terms $P_{0}^{X}$ and $P_{1}^{X}$ in a minimal projective presentation of $X$ as
\begin{align}
P_{0}^{X}:=\bigoplus_{i\in Q_{\ver}} (e_{i}\Lambda)^{n_{i}},\ P_{1}^{X}:=\bigoplus_{i\in Q_{\ver}} (e_{i}\Lambda)^{m_{i}}, \notag
\end{align} 
where $n_{i}$ and $m_{i}$ are multiplicities of the indecomposable projective $\Lambda$-module corresponding to $i\in Q_{\ver}$.
We denote by $Q^{X}$ the full subquiver of $Q^{\s}$ with
\begin{align}
Q_{\ver}^{X}:=\{ i^{+}\in Q^{\s}_{\ver}\mid n_{i}\neq 0\}\coprod\{ i^{-}\in Q^{\s}_{\ver} \mid m_{i}\neq 0\}. \notag
\end{align}
Then, the condition $\add P_{0}^{X}\cap \add P_{1}^{X}=0$ is satisfied if and only if $Q^{X}$ is a single subquiver of $Q^{\s}$.
In particular, if $X$ is $\tau$-rigid, then $Q^{X}$ is a single subquiver of $Q^{s}$ by Proposition \ref{2.2}(2).

Now we are ready to prove Theorem \ref{3.1}.
For any full subquiver $Q^{\prime}$ of $Q^{\s}$, let 
\begin{align}
\trigid(\Ma,Q^{\prime})
:&=\trigid\Ma\cap \Im L_{e_{Q^{\prime}}}\notag\\
&=\{ X\in\trigid\Ma\mid \add(P_{0}^{X}\oplus P_{1}^{X})\subseteq \add(\bigoplus_{i\in Q_{0}^{\prime}}e_{i}\Ma)\}\notag\\
\trigid^{\m}(\Ma,Q^{\prime})
&=\{ X\in\trigid\Ma\mid \add(P_{0}^{X}\oplus P_{1}^{X})=\add(\bigoplus_{i\in Q_{0}^{\prime}}e_{i}\Ma)\},\notag
\end{align}
where $e_{Q^{\prime}}:=\sum_{i\in Q_{\ver}^{\prime}}e_{i}$ and $L_{e_{Q^{\prime}}}$ is the functor in Subsection \ref{subsec2.1}.
We denote by $\trigid_{\np}\Lambda$ the subset of $\trigid\Lambda$ consisting of nonprojective $\Lambda$-modules.
\begin{proof}[Proof of Theorem \ref{3.1}]
(1)$\Leftrightarrow$(2):
First we claim that the functor $F:\mod\Lambda\xto{}\mod\Ma$ induces a bijection
\begin{align}\label{eq4}
\trigid_{\np} \Lambda \longrightarrow \bigcup_{Q^{\prime}\in\S}\trigid_{\np}(\Ma,Q^{\prime}). \tag{$\flat$}
\end{align}
Indeed, by Proposition \ref{2.7}(2) and Theorem \ref{3.2}, $X$ is an indecomposable nonprojective $\tau$-rigid $\Lambda$-module 
if and only if $FX$ is an indecomposable nonprojective $\tau$-rigid $\Ma$-module satisfying 
$\add P_{0}^{X}\cap \add P_{1}^{X}=0$, or equivalently $Q^{X}\in \S$.
In this case, since $P_{0}^{FX}\oplus P_{1}^{FX}\in \add e_{Q^{X}}\Ma$ holds by Lemma \ref{3.4}, we have $FX\in \Im L_{e_{Q^{X}}}$ by Lemma \ref{2.3}(2).
Hence the claim follows from that the functor $F$ induces a stable equivalence $\smod\Lambda\rightarrow \smod\Ma$ by Proposition \ref{2.7}(1).

Next, by Proposition \ref{2.4}, for any full subquiver $Q^{\prime}$ of $Q^{\s}$, we have bijections
\begin{align}
\trigid(\Ma,Q^{\prime}) \longleftrightarrow \trigid (e_{Q^{\prime}}\Ma e_{Q^{\prime}}). \notag
\end{align}
Since $Q^{\prime}$ is bipartite, 
there is an isomorphism $e_{Q^{\prime}}\Ma e_{Q^{\prime}}\simeq KQ^{\prime}$.
Since there are only finitely many single subquiver of $Q^{\s}$, we have that
$\Lambda$ is $\tau$-rigid-finite if and only if $KQ^{\prime}$ is $\tau$-rigid-finite for every single subquiver $Q^{\prime}$ of $Q^{\s}$. 
Hence the assertion follows from Theorem \ref{2.6}.
\end{proof}

By the proof of Theorem \ref{3.1}, we have the following corollary.
We denote by $\S^{+}$ the set of all connected single subquivers of $Q^{\s}$ except the quivers with exactly one vertex $i^{-}$ for any $i\in Q_{0}$.
Let 
\begin{align}
\trigid^{\m}KQ^{\prime}:=\{ M\in\trigid KQ^{\prime}\mid \add(P_{0}^{M}\oplus P_{1}^{M})=\add(KQ^{\prime})\} \notag
\end{align}
for a quiver $Q^{\prime}$.
\begin{corollary}\label{3.5}
There is a bijection
\begin{align}
\trigid\Lambda \longrightarrow\coprod_{Q^{\prime}\in\S^{+}}\trigid^{\m}KQ^{\prime} \notag
\end{align}
given by $X\mapsto (FX)e_{Q^{X}}$.
In particular, the cardinality of $\trigid\Lambda$ is 
\begin{align}
|\trigid\Lambda|=\sum_{Q^{\prime}\in\S^{+}}|\trigid^{\m}KQ^{\prime}|.\notag
\end{align}
\end{corollary}
\begin{proof}
Since the map $e_{i}\Lambda\mapsto e_{i^{+}}\Ma$ gives a bijection between 
the set of isomorphism classes of indecomposable projective $\Lambda$-modules and that of indecomposable $\Ma$-modules corresponding to the vertex $i^{+}\in Q_{0}^{+}$,
the functor $F$ induces a bijection
\begin{align}
\trigid\Lambda\longrightarrow\bigcup_{Q^{\prime}\in\S} \trigid(\Ma,Q^{\prime})\setminus\{ e_{i^{-}}\Ma\mid i\in Q_{0}\} \notag
\end{align}
by (\ref{eq4}).
Then, for any $X\in\trigid\Lambda$, we have $FX\in \trigid^{\m}(\Ma,Q^{X})$ with $Q^{X}\in\S^{+}$.
Moreover, we have 
\begin{align}
\bigcup_{Q^{\prime}\in\S} \trigid(\Ma,Q^{\prime})\setminus\{ e_{i^{-}}\Ma\mid i\in Q_{0}\}
&=\coprod_{Q^{\prime}\in\S} \trigid^{\m}(\Ma,Q^{\prime})\setminus\{ e_{i^{-}}\Ma\mid i\in Q_{0}\}\notag\\
&=\coprod_{Q^{\prime}\in\S^{+}} \trigid^{\m}(\Ma,Q^{\prime}). \notag
\end{align}
Hence the assertion follows from that the map $M\mapsto M_{e^{Q^{\prime}}}$ gives a bijection 
\begin{align}
\trigid^{\m}(\Ma, Q^{\prime})\longrightarrow \trigid^{\m}KQ^{\prime} \notag
\end{align}
by Proposition \ref{2.4}.
\end{proof}
\section{Applications and examples}
In this section, we give applications and examples of results in previous section.
As an immediate consequence of Theorem \ref{3.2}, we have the following two corollaries.
\begin{corollary}
Let $\Lambda$ be a representation-finite algebra with radical square zero and $X$ a $\Lambda$-module.
Then $X$ is a $\tau$-rigid $\Lambda$-module if and only if $\add P_{0}^{X}\cap \add P_{1}^{X}=0$.
\end{corollary}
\begin{proof}
The `only if' part follows from Proposition \ref{2.2}(2).
We show the `if' part.
Since $\Lambda$ is representation-finite, $\Ma$ is a finite product of path algebras with Dynkin quivers by Proposition \ref{2.7}(3).
Hence $FX$ is a $\tau$-rigid $\Ma$-module by Proposition \ref{2.5}(1).
Thus, if $\add P_{0}^{X}\cap \add P_{1}^{X}=0$ holds, then $X$ is $\tau$-rigid by Theorem \ref{3.2}.
\end{proof}
We give a positive answer to a question posed by Zhang \cite{Zh}.
\begin{corollary}\label{4.1}
Let $\Lambda$ be an algebra with radical square zero.
If every indecomposable $\Lambda$-module is $\tau$-rigid, then $\Lambda$ is representation-finite.
\end{corollary}
\begin{proof}
Assume that $\Lambda$ is not representation-finite.
By Proposition \ref{2.7}(3), the separated quiver contains a non-Dynkin quiver as a subquiver.
By Proposition \ref{2.5}(2), there exists an indecomposable $\Ma$-module $M$ which is not rigid.
By Proposition \ref{2.7}(1), there exists an indecomposable nonprojective $\Lambda$-module $X$ such that $FX\simeq M$.
The $\Lambda$-module $X$ is not $\tau$-rigid by Theorem \ref{3.2}.
\end{proof}
At the end of this paper, we apply our main results to the following algebras which associate with Brauer graph algebras.
We start with the following observation.
\begin{proposition}\cite[Corollary 3.7]{Ad}\label{4.2}
Let $\Lambda$ be a ring-indecomposable non-semisimple symmetric algebra and $\bLambda:=\Lambda/\soc\Lambda$.
Then there is a bijection
\begin{align}
\trigid\Lambda \longrightarrow \trigid\bLambda \notag
\end{align}
given by $X\mapsto X\otimes_{\Lambda}\bLambda$.
In particular, $\Lambda$ is $\tau$-rigid-finite if and only if $\bLambda$ is $\tau$-rigid-finite.
\end{proposition}
Note that, for every indecomposable projective $\Lambda$-module $P$, 
the module $P/\soc P$ is a $\tau$-rigid $\bLambda$-module but not a $\tau$-rigid $\Lambda$-module. 

First, we give a classification of indecomposable $\tau$-rigid modules over a multiplicity-free Brauer line algebra.
This is a special case of results in \cite{AZ} and \cite{AAC}.
We denote by $\ind \Lambda$  the set of isomorphism classes of indecomposable $\Lambda$-modules.
\begin{corollary}
Let $Q$ be the following quiver:
\begin{align}
\xymatrix{
1\ar@<0.6mm>[r]^{\alpha_{1}}&2\ar@<0.6mm>[l]^{\beta_{1}}\ar@<0.6mm>[r]^{\alpha_{2}}&3\ar@<0.6mm>[l]^{\beta_{2}}\ar@<0.6mm>[r]^{\alpha_{3}}&
\cdots\ar@<0.6mm>[l]^{\beta_{3}}\ar@<0.6mm>[r]^{\alpha_{n-2}}&n-1\ar@<0.6mm>[l]^{\beta_{n-2}}\ar@<0.6mm>[r]^{\alpha_{n-1}}&n\ar@<0.6mm>[l]^{\beta_{n-1}}}\notag
\end{align}
\begin{itemize}
\item[(1)] Let $\Lambda$ be an algebra with radical square zero whose quiver is $Q$.
Then $\Lambda$ is a representation-finite algebra with 
\begin{align}
\trigid\Lambda=\ind \Lambda . \notag 
\end{align}
\item[(2)] Let $\Gamma$ be a multiplicity-free Brauer line algebra, that is, $\Gamma\simeq KQ/I$, where 
\begin{align}
I=\langle \alpha_{1}\beta_{1}\alpha_{1}, \beta_{n-1}\alpha_{n-1}\beta_{n-1}, 
\alpha_{i}\alpha_{i+1}, \beta_{i+1}\beta_{i}, \beta_{i}\alpha_{i}-\alpha_{i+1}\beta_{i+1}\mid i=1,2,\cdots,n-2\rangle. \notag
\end{align}
Then we have 
\begin{align}
\trigid \Gamma =\ind \Gamma \setminus \{ e_{i}\Gamma/\soc (e_{i}\Gamma) \mid i\in Q_{\ver}\}. \notag
\end{align}
\item[(3)] The cardinalities of $\trigid\Lambda$ and $\trigid\Gamma$ are
\begin{align}
|\trigid\Lambda|=|\trigid\Gamma|=n^{2}.\notag
\end{align}
\end{itemize}
\end{corollary}
\begin{proof}
(1) Since the underlying graph of the separated quiver $Q^{\s}$ is a disjoint union of the following two Dynkin graphs of type $A$
\begin{align}
\xymatrix@C=8mm{
{\begin{smallmatrix}{1^{+}}\end{smallmatrix}}\ar@{-}[r]&{\begin{smallmatrix}{2^{-}}\end{smallmatrix}}&\ar@{-}[l]\ar@{-}[r]\cdots
&{\begin{smallmatrix}{n-1^{-\epsilon}}\end{smallmatrix}}\ar@{-}[r]&{\begin{smallmatrix}{n^{\epsilon}}\end{smallmatrix}}}\ \ 
\xymatrix@C=8mm{
{\begin{smallmatrix}{1^{-}}\end{smallmatrix}}&{\begin{smallmatrix}{2^{+}}\end{smallmatrix}}\ar@{-}[l]\ar@{-}[r]&\cdots\ar@{-}[r]
&{\begin{smallmatrix}{n-1^{\epsilon}}\end{smallmatrix}}\ar@{-}[r]&{\begin{smallmatrix}{n^{-\epsilon}}\end{smallmatrix}}}\notag
\end{align}
where $\epsilon=-$ if $n$ is even and $\epsilon=+$ if $n$ is odd, 
$\Lambda$ is representation-finite by Proposition \ref{2.7}.
Moreover, for any indecomposable $\Lambda$-module $X$, 
the $\Ma$-module $FX$ is rigid by Proposition \ref{2.5}(1), or equivalently $\tau$-rigid.
Moreover, we have $\add P_{0}^{X}\cap \add P_{1}^{X}=0$ by Lemma \ref{3.4}.
Hence, every indecomposable $\Lambda$-module is always $\tau$-rigid by Theorem \ref{3.2}.

(2) Since $\Gamma$ is a symmetric algebra, there is a bijection
\begin{align}
\trigid\Gamma \longrightarrow \trigid\overline{\Gamma} \notag 
\end{align}
by Proposition \ref{4.2}.
Since $\overline{\Gamma}$ is isomorphic to $\Lambda$, the assertion follows from (1).

(3) By Corollary \ref{3.5}, we have
\begin{align}
|\trigid\Lambda|=|\trigid\Gamma|=\sum_{Q^{\prime}\in\S^{+}}|\trigid^{\m}KQ^{\prime}|. \notag
\end{align}
By (1), all single subquiver in $\S^{+}$ with $l>1$ (respectively, $l=1$) vertices are exactly $2(n-l+1)$ (respectively, $n$) Dynkin quivers of type $A$. 
Since, for each Dynkin quiver $Q^{\prime}$ of type $A$, we have $|\trigid^{\m} KQ^{\prime}|=1$, the cardinalities are 
\begin{align}
\displaystyle|\trigid\Lambda|=|\trigid\Gamma|=n+2\sum_{l=2}^{n}(n-l+1)=n^{2}. \notag
\end{align}
\end{proof}
Finally, we give an example of $\tau$-rigid-finite algebras which is not representation-finite.
\begin{corollary}\label{4.3}
Let $Q$ be the following quiver:
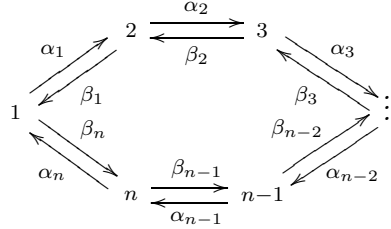
\begin{align}
\xymatrix@C=10mm@R=5mm{
&\begin{smallmatrix}2\end{smallmatrix}\ar@<1mm>[ld]^{\beta_{1}}\ar@<1mm>[r]^{\alpha_{2}}
&\begin{smallmatrix}3\end{smallmatrix}\ar@<1mm>[l]^{\beta_{2}}\ar@<1mm>[rd]^{\alpha_{3}}&\\
\begin{smallmatrix}1\end{smallmatrix}\ar@<1mm>[ru]^{\alpha_{1}}\ar@<1mm>[rd]^{\beta_{n}}&
&&\vdots \ar@<1mm>[lu]^{\beta_{3}}\ar@<1mm>[ld]^{\alpha_{n-2}}\\
&\begin{smallmatrix}n\end{smallmatrix}\ar@<1mm>[r]^{\beta_{n-1}}\ar@<1mm>[lu]^{\alpha_{n}}
&\begin{smallmatrix}n-1\end{smallmatrix}\ar@<1mm>[ru]^{\beta_{n-2}}\ar@<1mm>[l]^{\alpha_{n-1}}&
}\notag
\end{align}
\begin{itemize}
\item[(1)] Let $\Lambda$ be an algebra with radical square zero whose quiver is $Q$.
Then the following hold.
\begin{itemize}
\item[(a)] $\Lambda$ is not representation-finite.
\item[(b)] $\Lambda$ is $\tau$-rigid-finite if and only if $n$ is odd.
\end{itemize}
\item[(2)] Let $\Gamma$ be a multiplicity-free Brauer cyclic graph algebra, that is, $\Gamma\simeq KQ/I$, where 
\begin{align}
I=\langle \alpha_{n}\alpha_{1}, \beta_{1}\beta_{n}, \beta_{n}\alpha_{n}-\alpha_{1}\beta_{1}, 
\alpha_{i}\alpha_{i+1}, \beta_{i+1}\beta_{i}, \beta_{i}\alpha_{i}-\alpha_{i+1}\beta_{i+1}  \mid i=1,2,\cdots,n-1\rangle. \notag
\end{align}
Then $\Gamma$ is $\tau$-rigid-finite if and only if $n$ is odd.
\item[(3)] Assume that $n$ is odd. Then the cardinalities of $\trigid\Lambda$ and $\trigid\Gamma$ are
\begin{align}
|\trigid\Lambda|=|\trigid\Gamma|=2n^{2}-n. \notag
\end{align}
\end{itemize}
\end{corollary}
\begin{proof}
(1) The separated quiver $Q^{\s}$ is one of the following quivers:
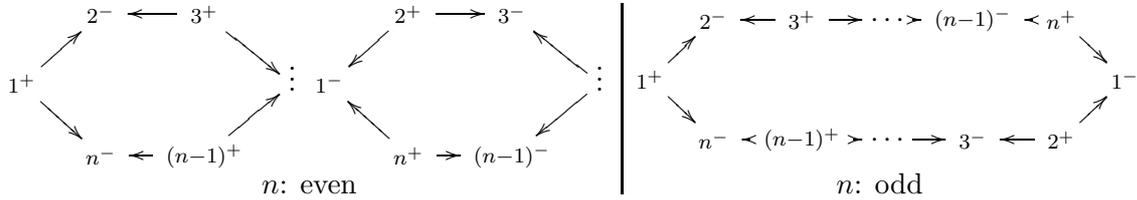
\begin{align}
\begin{picture}(300,75)(0,0)
\put(-65,65){$\xymatrix@C=3mm@R=3mm{
&\begin{smallmatrix}2^{-}\end{smallmatrix}&\begin{smallmatrix}3^{+}\end{smallmatrix}\ar[l]\ar[rd]&\\
\begin{smallmatrix}1^{+}\end{smallmatrix}\ar[ru]\ar[rd]&&&\vdots\\
&\begin{smallmatrix}n^{-}\end{smallmatrix}&\begin{smallmatrix}(n-1)^{+}\end{smallmatrix}\ar[ru]\ar[l]&
}$}
\put(50,65){$\xymatrix@C=3mm@R=3mm{
&\begin{smallmatrix}2^{+}\end{smallmatrix}\ar[ld]\ar[r]&\begin{smallmatrix}3^{-}\end{smallmatrix}&\\
\begin{smallmatrix}1^{-}\end{smallmatrix}&&&\vdots \ar[lu]\ar[ld]\\
&\begin{smallmatrix}n^{+}\end{smallmatrix}\ar[r]\ar[lu]&\begin{smallmatrix}(n-1)^{-}\end{smallmatrix}&
}$}
\thicklines
\put(170,0){\line(0,1){72}}
\put(170,63){$\xymatrix@C=1mm@R=3mm{
&\begin{smallmatrix}2^{-}\end{smallmatrix}&\begin{smallmatrix}3^{+}\end{smallmatrix}\ar[l]\ar[r]&\cdots\ar[r]&
\begin{smallmatrix}(n-1)^{-}\end{smallmatrix}&\begin{smallmatrix}n^{+}\end{smallmatrix}\ar[l]\ar[rd]&\\
\begin{smallmatrix}1^{+}\end{smallmatrix}\ar[ru]\ar[rd]&&&&&&\begin{smallmatrix}1^{-}\end{smallmatrix}\\
&\begin{smallmatrix}n^{-}\end{smallmatrix}&\begin{smallmatrix}(n-1)^{+}\end{smallmatrix}\ar[l]\ar[r]&
\cdots\ar[r]&\begin{smallmatrix}3^{-}\end{smallmatrix}&\begin{smallmatrix}2^{+}\end{smallmatrix}\ar[l]\ar[ru]&
}$}
\put(35,0){$n$: even}
\put(250,0){$n$: odd}
\end{picture}\notag
\end{align}
Thus $\Lambda$ is not representation-finite by Proposition \ref{2.7}(3).

If $n$ is odd, then every single subquiver is a disjoint union of Dynkin quivers.
Thus $\Lambda$ is $\tau$-rigid-finite by Theorem \ref{3.1}.
On the other hand, if $n$ is even, then two connected components are non-Dynkin single subquivers.
Thus $\Lambda$ is not $\tau$-rigid-finite by Theorem \ref{3.1}.

(2) Since $\Gamma$ is a symmetric algebra, by Proposition \ref{4.2}, we have only to claim that 
$\overline{\Gamma}$ is $\tau$-rigid-finite if and only if $n$ is odd.
Indeed, since $\overline{\Gamma}$ is isomorphic to $\Lambda$, the claim follows from (1).

(3) By Corollary \ref{3.5}, we have
\begin{align}
|\trigid\Lambda|=|\trigid\Gamma|=\sum_{Q^{\prime}\in\S^{+}}|\trigid^{\m}KQ^{\prime}|. \notag
\end{align}
By (1), all single subquiver in $\S^{+}$ with $l>1$ (respectively, $l=1$) vertices are exactly $2n$ (respectively, $n$) Dynkin quivers of type $A$. 
Since, for each Dynkin quiver $Q^{\prime}$ of type $A$, we have $|\trigid^{\m} KQ^{\prime}|=1$, the cardinalities are  
\begin{align}
\displaystyle|\trigid\Lambda|=|\trigid\Gamma|=n+2n(n-1)=2n^{2}-n. \notag
\end{align}
\end{proof}

\end{document}